\documentclass[12pt,leqno,a4paper]{amsart}
 
\usepackage{latexsym}
\usepackage{amssymb,amsmath}
\usepackage{graphics}
\usepackage{url}
\usepackage[vcentermath]{youngtab}
  
\usepackage{booktabs}  
\usepackage{amssymb,enumerate}
\newcommand{\fA}{{\mathfrak{A}}}
\newcommand{\fS}{{\mathfrak{S}}}

\let\la=\lambda

\def\irr#1{{\rm Irr}(#1)}
\def\syl#1#2{{\rm Syl}_{#1}(#2)}

\newtheorem{thm}{Theorem}[section]
\newtheorem{lem}[thm]{Lemma}

\newtheorem{cor}[thm]{Corollary}
\newtheorem{prop}[thm]{Proposition}

\newtheorem*{thmA}{Theorem A}

\newtheorem*{conjD}{Conjecture}

\theoremstyle{definition}

\newtheorem{defn}[thm]{Definition}

\newtheorem{notation}[thm]{Notation}

\theoremstyle{remark}

\begin{document}

\title[]{A note on restriction of Characters of Alternating groups to Sylow
Subgroups}

\author[Eugenio Giannelli]{Eugenio Giannelli}
\address[E. Giannelli]{Trinity Hall, University of Cambridge, Trinity Lane, CB21TJ, UK}
\email{eg513@cam.ac.uk}

\thanks{The first author's research was funded by Trinity Hall, University of Cambridge.}
 
\begin{abstract}
We restrict irreducible characters of alternating groups
of degree 
divisible by $p$  to
their Sylow $p$-subgroups and study the number 
of linear constituents.
\end{abstract}



\maketitle


\section{Introduction}   \label{sec:intro}

Let $p$ be a prime number and let $G$ be a finite group. We denote by $\mathrm{Irr}(G)$ the set of ordinary irreducible characters of $G$ and we let $P$ be a Sylow $p$-subgroup of $G$. 
It seems a very natural problem to study the restriction to $P$ of irreducible characters of $G$. Our interest in this line of research was initially motivated by the long-standing McKay Conjecture (see \cite{McKay} for the original statement and/or \cite[Section 2]{Malle} for a complete survey of the state of the art on this problem). 
For some special classes of groups it has been possible to show that the restrictions to a Sylow $p$-subgroup of irreducible characters of degree coprime to $p$ admit a unique linear (i.e. coprime to $p$) constituent. This allowed to show stronger forms of the McKay conjecture. This happens for example in \cite{G} and \cite{N}.

In \cite{GN} we studied the number of linear constituents of the restricted character $\chi_P$ for any $\chi\in\mathrm{Irr}(G)$ such that $p$ divides $\chi(1)$. 
As a consequence of the evidences collected, 
the following conjecture was proposed.

 \begin{conjD}
 Suppose that $\chi \in \irr G$ has degree divisible by $p$, and let $P \in \syl pG$.
 If $\chi_P$ has a linear constituent, then $\chi_P$ has at least $p$ different linear constituents.
 \end{conjD}
 
The above statement has been verified for various classes of groups in \cite{GN}. In particular, when $G$ is the symmetric group $\fS_n$ it is shown that the restriction to a Sylow $p$-subgroup of any irreducible character of degree divisible by $p$ has at least $p$ distinct linear constituents. 
 
The aim of this note is to prove that the above statement holds for alternating groups. 
In particular for any natural number $n$ we let $\fA_n$ be the alternating group of degree $n$ and we denote by $Q_n$ a Sylow $p$-subgroup of $\fA_n$. 
The main result of this note is the following.

\begin{thmA}
Let $n$ be a natural number and let $p$ be a prime. 
If $\chi \in {\rm Irr}(\fA_n)$  has 
degree divisible by $p$, then
the restriction $\chi_{Q_n}$ has at least $p$ different  
linear constituents.
\end{thmA}
 
When $p$ is an odd prime, Theorem A follows quite directly from the corresponding statement for symmetric groups (see Section \ref{sec: 3.1} below). 
On the other hand in order to prove Theorem A for $p=2$, we will need to use some new algebraic and combinatorial ideas (this is done in Section \ref{sec: 3.2}).  

 \section{Background}

We start by recalling some very basic combinatorial definitions and notation in the framework of the representation theory of symmetric groups. We refer the reader to \cite{James} or \cite{OlssonBook} for a more detailed account. 

A composition $\lambda=(\lambda_1,\lambda_2,\dots ,\lambda_s)$ is a finite  sequence of positive integers. We say that $\lambda_i$ is a part of $\lambda$ and that $\lambda$ is a composition of $|\lambda|=\sum\lambda_i$. We say that $\lambda$ is a partition if its parts are non-increasingly ordered and we denote by $\lambda'$ the conjugate partition of $\lambda$. The Young diagram associated to $\lambda$ is the set $[\lambda]:=\{(i,j)\in{\mathbb N}\times{\mathbb N}\mid 1\leq i\leq s,1\leq j\leq\lambda_i\}$.
For any natural number $n$ we denote by $\mathcal{P}(n)$ (respectively $\mathcal{C}(n)$) the set of partitions (respectively compositions) of $n$. 
We will sometimes write $\lambda\vdash n$ if $\lambda\in\mathcal{P}(n)$. 
Following the notation used in \cite{OlssonBook}, given $(i,j)\in [\lambda]$ we denote by $H_{i,j}(\lambda)$ the corresponding $(i,j)$-hook. Moreover we let $h_{i,j}(\lambda)=|H_{i,j}(\lambda)|$. Finally, for any natural number $e$ we denote by $\mathcal{H}_e(\lambda)$ the subset of $[\lambda]$ defined by 
$$\mathcal{H}_e(\lambda)=\{(i,j)\in [\lambda]\ :\ e\mid h_{i,j}(\lambda)\}.$$
Given an $e$-hook $h:=H_{i,j}(\lambda)$ we denote by $\lambda - h$ the partition of $n-e$ obtained from $\lambda$ by removing its $(i,j)$-rim hook (we refer the reader to \cite{OlssonBook} for the formal definition). 
We denote by $C_e(\lambda)$ and $Q_e(\lambda)$ the $e$-core and the $e$-quotient of $\lambda$, respectively. Finally we let $w_e(\lambda)$ be the $e$-weight of $\lambda$.

Let $\mu$ and $\lambda$ be partitions. We say that $\mu$ is a \textit{subpartition} of $\lambda$, written $\mu\subseteq\lambda$, if $\mu_i\leq\lambda_i$, for all $i\geq1$. When this occurs, we let
$[\lambda\smallsetminus\mu]=\{(i,j)\in{\mathbb N}\times{\mathbb N}\mid 1\leq i\leq s,\mu_i<j\leq\lambda_i\}.$ 
This is called a skew Young diagram.

\subsection{Characters of $\fS_n$ and $\fA_n$}
Irreducible characters of $\fS_n$ are naturally labelled by partitions of $n$. Given $\lambda\in\mathcal{P}(n)$ we denote by $\chi^\lambda$ the corresponding element of $\mathrm{Irr}(\fS_n)$. 
If $\lambda\neq \lambda'$, then we denote by $\phi^\lambda$ the irreducible character of $\mathfrak{A}_n$ defined by $\phi^\lambda=(\chi^\lambda)_{\mathfrak{A}_n}$. On the other hand, if $\lambda=\lambda'$ then we let $\phi^\lambda_+$ and $\phi^\lambda_-$ be the two irreducible characters of $\mathfrak{A}_n$ such that $(\chi^\lambda)_{\mathfrak{A}_n}=\phi^\lambda_++\phi^\lambda_-$. Clearly we have that $(\phi_+^\lambda)^{\sigma}=\phi^\lambda_-$, for any $\sigma\in\fS_n\smallsetminus\fA_n$.

\medskip

A key ingredient in our proof will be a rather sophisticated use of the Littlewood-Richardson rule (see \cite[Chapter 16]{James}). For the convenience of the reader we recall this here.  

\begin{defn}
Let ${\mathcal A}=a_1,\dots,a_k$ be a sequence of positive integers. The type of $\mathcal{A}$ is the sequence of non-negative integers $m_1,m_2,\dots$ where $m_i$ is the number of occurrences of $i$ in $a_1,\dots,a_k$. We say that $\mathcal{A}$ is a \textit{reverse lattice sequence} if the type of its prefix $a_1,\dots,a_j$ is a partition, for all $j\geq1$. 
%

Let $\mu\vdash n$ and $\nu\vdash m$. The outer tensor product $\chi^\mu\times\chi^\nu$ is an irreducible character of ${\mathfrak S}_n\times{\mathfrak S}_m$. Inducing this character to ${\mathfrak S}_{n+m}$ we may write
$$
(\chi^\mu\times\chi^\nu)^{{\mathfrak S}_{n+m}}=\sum_{\lambda\vdash(n+m)}C_{\mu,\nu}^\lambda\chi^\lambda.
$$
The \textit{Littlewood-Richardson rule} asserts that $C_{\mu,\nu}^\lambda$ is zero if $\mu\not\subseteq\lambda$ and otherwise equals the number of ways to replace the nodes of the skew diagram $[\lambda\smallsetminus\mu]$ by natural numbers such that: 
the numbers are weakly increasing along rows;  
the numbers are strictly increasing down the columns;
the sequence obtained by reading the numbers from right to left and top to bottom is a reverse lattice sequence of type $\nu$.
We call any such configuration a \textit{Littlewood-Richardson configuration of type} $\nu$ for $[\lambda\smallsetminus\mu]$.
\end{defn}

The Murnaghan-Nakayama rule (see \cite[2.4.7]{JK}) allows to recursively compute the character table of symmetric groups. 
The following fact (see \cite[2.7.33]{JK}) is a very useful consequence of this. 

\begin{lem}\label{lem:farahat}
Let $\la\vdash n$, $Q_e(\la)=(\la_0,\ldots,\la_{e-1})$.
Assume that $w\ge w_e(\la)$, and let $\rho \in \fS_n$
be  the product of $w$ $e$-cycles;
let $\gamma \in \fS_n$ act on the fixed points of $\rho$.
Then
$$\chi^{\la}(\rho\gamma)=
\begin{cases}
\pm \binom{w}{|\la_0|,|\la_1|,\ldots,|\la_{e-1}|}\chi^{C_e(\la)}(\gamma)\prod_{i=0}^{e-1} \chi^{\la_i}(1) &  \text{if } w=w_e(\la) \\
 \rm{0} &  \text{if } w>w_e(\la)
\end{cases} \:.$$
\end{lem}

\section{The proof of Theorem A}

Let $n$ be a natural number and let $P_n$ be a Sylow $p$-subgroup of $\fS_n$. 
Denote by $Q_n=P_n\cap\fA_n$ the unique Sylow $p$-subgroup of $\fA_n$ contained in $P_n$. Clearly $Q_n\trianglelefteq P_n$. We recall that if $n=a_1p^{n_1}+a_2p^{n_2}+\cdots+a_kp^{n_k}$ is the $p$-adic expansion of $n$ then $P_n=(P_{p^{n_1}})^{\times a_1}\times (P_{p^{n_2}})^{\times a_2}\times\cdots\times (P_{p^{n_k}})^{\times a_k}.$

\subsection{Odd primes}\label{sec: 3.1} Let us start by assuming that $p$ is an odd prime. 
In this case we clearly have that $Q_n=P_n$. 
If $n=\sum_{j=1}^ka_jp^{n_j}$ is the $p$-adic expansion of $n$ then we denote by $\omega_n$ a chosen element of $P_n$ obtained as the product of $a_j$ cycles of length $p^{n_j}$, for all $j\in\{1,\ldots, k\}$.
The following Lemma is a consequence of \cite[Lemma 3.11]{GN}. 

\begin{lem}\label{lem:omega_n}
Let $n, p$ and $\omega_n$ be as above. Then
\begin{itemize}
\item[(i)] $\theta(\omega_n)$ is a $p$-th root of unity for every linear character $\theta$ of $P_{n}$.
\item[(ii)] $\delta(\omega_n)=0$ for all  $\delta\in\mathrm{Irr}(P_{n})$ such that $p\ |\ \delta(1)$.
\end{itemize}
\end{lem}

The following result is a corollary of \cite[Theorem 3.1]{GN}.

\begin{lem}\label{lem:uniquelinear}
Let $\phi\in\mathrm{Irr}(\fA_n)$. 
Then the restriction of $\phi$ to $\fA_n$ has a linear constituent. 
\end{lem}
\begin{proof}
Let $\lambda$ be a partition of $n$ such that $\phi$ is an irreducible constituent of $(\chi^\lambda)_{\fA_n}$. Let $\theta$ be a linear constituent of $(\chi^\lambda)_{P_n}$ (this exists by  \cite[Theorem 3.1]{GN}).
If $\lambda\neq \lambda'$ then $\theta$ is a constituent of $\phi_{P_n}$.
Suppose that $\lambda=\lambda'$ and (without loss of generality) that $\phi=\phi^\lambda_+$. Let $\sigma\in N_{\fS_n}(P_n)\smallsetminus \fA_n$. Such a $\sigma$ exists (choose for example an element in $N_{\fS_n}(P_n)$ of cycle type given by the product of $p^{k-1}$ cycles of length $p-1$, where $p^k$ is a $p$-adic digit of $n$). 
Then if we assume that $\theta$ is not a constituent of $\phi_{P_n}$ we deduce that $\theta$ is a constituent of $(\phi^\lambda_-)_{P_n}$. Hence $\theta^\sigma$ is a linear constituent of $((\phi^\lambda_-)_{P_n})^{\sigma}=((\phi^\lambda_-)^{\sigma})_{P_n}=\phi_{P_n}$.
\end{proof}

We are ready to prove Theorem A of the introduction for odd primes. 

\begin{prop}
Let $p$ be an odd prime, let $n$ be a natural number and let $\phi\in\mathrm{Irr}(\fA_n)$ be such that $p$ divides $\phi(1)$. Then the restriction of $\phi$ to $P_n$ has at least $p$ 
distinct linear constituents. 
\end{prop}
\begin{proof}
Let $\lambda$ be a partition of $n$ such that  $\phi$ is an irreducible constituent of $(\chi^\lambda)_{\fA_n}$. If $\lambda\neq\lambda'$ then the statement follows from \cite[Theorem A]{GN}.
Assume now that  $\lambda=\lambda'$ and (without loss of generality) $\phi=\phi^\lambda_+$. Let $n=\sum_{j=0}^ka_jp^j$ be the $p$-adic expansion of $n$ and let 
$\omega_n\in P_{n}$ be the product of $a_j$ $p^j$-cycles for $j\in\{0,1,\ldots, k\}$. 
Since $p$ divides $\phi(1)$ we have that $p$ divides $\chi^\lambda(1)$ and therefore that $\chi^\lambda(\omega_n)=0$. This follows from the description of irreducible characters of symmetric groups whose degree is not divisible by $p$ (see for instance \cite{Mac} or \cite[Proposition 6.4]{OlssonBook}). Moreover, by \cite[Theorem 2.5.13]{JK} we obtain that $\phi(\omega_n)=\frac{\chi^\lambda(g)}{2}=0$. Suppose for a contradiction that $\phi_{P_n}$ has $\ell$ distinct linear constituents $\theta_1, \ldots, \theta_\ell$, for some $1\leq \ell<p$. From Lemma \ref{lem:omega_n} we deduce that 
$$0=\phi(\omega_n)=c_1\theta_1(\omega_n)+\cdots c_\ell\theta_\ell(\omega_n)=c_1\zeta_1+\cdots c_\ell\zeta_\ell,$$
for some $p$-th roots of unity $\zeta_1, \ldots, \zeta_\ell$ and some positive natural numbers $c_1, \ldots, c_\ell$. 
This is a contradiction, since no $\mathbb{N}$-linear combination of $\zeta_1, \ldots, \zeta_\ell$ can equal $0$.
\end{proof}

\subsection{The prime $2$}\label{sec: 3.2}
To deal with the remaining $p=2$ case, we need to fix some more precise notation. 
Let $n\in\mathbb{N}$, for $k\in\{0,1,\ldots, n-1\}$ we let $g_k$ be the element of $\fS_{2^n}$ defined by $$g_k=\prod_{i=1}^{2^k}(i, i+2^k).$$
It is not too difficult to see that $P_{2^n}:=\left\langle g_0,\ldots, g_{n-1}\right\rangle$ is a Sylow $2$-subgroup of $\fS_{2^n}$. Moreover the element $\gamma_n\in P_{2^n}$ defined by $\gamma_n=g_1g_2\cdots g_{n-1}$ has cycle type $(2^{n-1}, 2^{n-1})$. In particular $\gamma_n\in\fA_{2^n}\cap P_{2^n}=Q_{2^n}$. 
Similarly, the element $\omega_n\in P_{2^n}$ defined by $\omega_n=g_0\gamma_n$ is a $2^n$-cycle.  

\begin{lem}\label{lem:2-root}
Let $\gamma_n$ and $\omega_n$ be the elements defined above and let $g\in\{\gamma_n,\omega_n\}$. Then:  
\begin{itemize}
\item[(i)] $\theta(g)\in\{-1,+1\}$ for every linear character $\theta$ of $P_{2^n}$.
\item[(ii)] $\delta(g)=0$ for all $\delta\in\mathrm{Irr}(P_{2^n})$ such that $\delta(1)$ is even.
\end{itemize}
\end{lem}
\begin{proof}
Let $g=\gamma_n$. It is well known that $P_{2^n}=P_{2^{n-1}}\wr C_2=(P_{2^{n-1}}\times P_{2^{n-1}})\rtimes C_2$. Let $\sigma$ be a generator of the top group $C_2$, acting on the base group by swapping the two direct factors. From this point of view we observe that $\gamma_n=(\gamma_{n-1}, 1; \sigma)$. To prove (i) we now proceed by induction on $n$. If $n=1$, the result is obviously true. Hence assume that $n\geq 2$ and let $\theta$ be a linear character of $P_{2^n}$. By Gallagher's Corollary (see \cite[6.17]{I}) we know that $\theta$ is the extension by an irreducible character $\psi$ of $C_2$ of an irreducible character $\phi\times\phi\in\mathrm{Irr}(P_{2^{n-1}}\times P_{2^{n-1}})$, for some linear character $\phi$ of $P_{2^{n-1}}$. Using \cite[Lemma 4.3.9]{JK}, it follows that $\theta(\gamma_n)=\phi(\gamma_{n-1})\psi(\sigma)$. Using the inductive hypothesis we conclude that $\phi(\gamma_{n-1})\in\{-1,+1\}$ and therefore that $\theta(\gamma_n)=\pm 1$.

The proof of case (ii) is again done by induction on $n$. We observe that if $\delta$ is an irreducible character of $P_{2^n}$ of even degree then either $\delta=(\phi_1\times\phi_2)^{P_{2^n}}$, for some $\phi_1,\phi_2\in\mathrm{Irr}(P_{2^{n-1}}\times P_{2^{n-1}})$ such that $\phi_1\neq\phi_2$; 
or otherwise $\delta$  is the extension by an irreducible character $\psi$ of $C_2$ of an irreducible character $\phi\times\phi\in\mathrm{Irr}(P_{2^{n-1}}\times P_{2^{n-1}})$, for some even-degree character $\phi$ of $P_{2^{n-1}}$. 
In the first case we have that $\delta(\gamma_n)=0$ because $\gamma_n\notin P_{2^{n-1}}\times P_{2^{n-1}}$. In the second case we can write again $\delta(\gamma_n)=\phi(\gamma_{n-1})\psi(\sigma)=0$, by induction. 
If $g=\omega_n$ then the result follows from \cite[Lemma 3.11]{GN}.
\end{proof}

It is useful to recall here the characterization of those partitions of a natural number n labelling irreducible characters of odd degree of symmetric groups. This was first observed in \cite{Mac}. As usual we adopt the symbol $\mathrm{Irr}_{2'}(G)$ to denote the set of irreducible characters of the finite group $G$ whose degree is odd. Moreover, we let $\mathrm{Lin}(G)$ be the notation for the set of linear characters of $G$.

\begin{thm}\label{thm:oddMac}
Let $n$ be a natural number with binary expansion $n=2^{n_1}+2^{n_2}+\cdots+2^{n_t}$, for some $n_1>\cdots >n_t\geq 0$ and let $\lambda\in\mathcal{P}(n)$. Then, $\chi^{\lambda}\in\mathrm{Irr}_{2'}(\fS_{n})$ if and only if there exists a unique removable $2^{n_1}$-hook $h_1$ in $\lambda$ and  $\chi^{(\lambda - h_1)}\in\mathrm{Irr}_{2'}(\fS_{n-2^{n_1}})$.
\end{thm}

For any $n\in\mathbb{N}$ we denote by $\mathcal{L}(n)$ the subset of $\mathcal{P}(n)$ consisting of \textit{hook partitions} (i.e. partitions of the form $(n-x, 1^x)$ for $0\leq x\leq n-1$). An useful consequence of Theorem \ref{thm:oddMac} is that $\chi^\lambda\in\mathrm{Irr}_{2'}(\fS_{2^n})$ if and only if $\lambda\in\mathcal{L}(n)$. 

In this article we will need a little bit of control on the maximal power of $2$ dividing the degrees of irreducible characters. We say that $\nu_2(n)=a$ if $n=2^am$, for some odd $m\in\mathbb{N}$.
The theory of $2$-core and $2$-quotient towers, introduced in \cite{Mac} and fully developed and beautifully explained in \cite{OlssonBook}, implies the following fact. 

\begin{lem}\label{lem: towers8}
Let $n$ be natural number and let $\lambda\in\mathcal{P}(2^n)$. Then: 
\begin{itemize}
\item[(i)] $\nu_2(\chi^\lambda(1))=1$ if and only if $\mathcal{H}_{2^n}(\lambda)=\emptyset$ and $|\mathcal{H}_{2^{n-1}}(\lambda)|=2$.
\item[(ii)] $\nu_2(\chi^\lambda(1))\geq 2$ if and only if $|\mathcal{H}_{2^{n-1}}(\lambda)|\leq 1$.
\end{itemize}
\end{lem}

The Murnaghan-Nakayama formula together with Theorem \ref{thm:oddMac} imply the following fact. 

\begin{cor}\label{cor:g lambda}
Let $n$ be a natural number with binary expansion $n=2^{n_1}+\cdots+2^{n_t}+a_02^0,$ for some $t\geq 2$,  $n_1>\cdots >n_t> 0$ and $a_0\in\{0,1\}$.
Let $\lambda\in\mathcal{P}(n)$ be such that $\chi^\lambda(1)$ is even. 
Then there exists $g_\lambda\in Q_n$ such that the following hold. 

\medskip

\begin{itemize}
\item[(i)] $\theta(g_\lambda)\in\{-1,+1\}$ for every linear character $\theta$ of $P_n$.
\item[(ii)] $\delta(g_\lambda)=0$ for all  $\delta\in\mathrm{Irr}(P_n)$ such that $\delta(1)$ is even.
\item[(iii)] $\chi^\lambda(g_\lambda)=0$.
\end{itemize}
\end{cor}

\begin{proof}
Since $\chi^\lambda(1)$ is even, Theorem \ref{thm:oddMac} implies that either $\mathrm{core}_{2^{n_1}}(\lambda)=\lambda$ (i.e. $\lambda$ does not have $2^{n_1}$-hooks) or
that there exists $r\in\{2,\ldots, t\}$ such that 
$$|\mathrm{core}_{2^{n_{j}}}(\lambda)|=|\mathrm{core}_{2^{n_{j-1}}}(\lambda)|-2^{n_j},\ \ \text{for all}\ \ j\in\{2, \ldots, r-1\},$$ and such that 
$\mathrm{core}_{2^{n_{r}}}(\lambda)=\mathrm{core}_{2^{n_{r-1}}}(\lambda)$ (i.e. $\mathrm{core}_{2^{n_{r-1}}}(\lambda)$ does not have $2^{n_r}$-hooks).

If $t$ is even, then we let $g_\lambda\in Q_n$ be any element of cycle type $(2^{n_1},2^{n_2},\ldots, 2^{n_t})$. 
Statements (i) and (ii) now follow from Lemma \ref{lem:2-root}. Moreover, statement (iii) is a consequence of the Murnaghan-Nakayama rule used together with Theorem \ref{thm:oddMac}. 

If $t$ is odd and $\mathrm{core}_{2^{n_1}}(\lambda)=\lambda$ then we let $g_\lambda\in Q_n$ be an element of cycle type $(2^{n_1},2^{n_2-1}, 2^{n_2-1},2^{n_3},\ldots, 2^{n_t})$.
Notice that we can always find an element of this form in $Q_n$ because $t\geq 3$. 
On the other hand, if $t$ is odd but $\mathrm{core}_{2^{n_1}}(\lambda)\neq\lambda$, 
then we choose $g_\lambda\in Q_n$ to be an element of cycle type $(2^{n_1-1},2^{n_1-1}, 2^{n_2},\ldots, 2^{n_r},\ldots,  2^{n_t})$.
Statements (i) and (ii) now follow from Lemma \ref{lem:2-root}. Statement (iii) is again a consequence of the Murnaghan-Nakayama rule used together with Theorem \ref{thm:oddMac}. 
\end{proof}

\begin{lem}\label{lem:lin1}
Let $n$ be a natural number, let $\chi\in\mathrm{Irr}(\fS_n)$ and let $\phi\in\mathrm{Irr}(\fA_n)$ be an irreducible constituent of $\chi_{\fA_n}$. Let $\theta\in\mathrm{Lin}(P_n)$ be a constituent of $\chi_{P_n}$. Then $\theta_{Q_n}$ is a constituent of $\phi_{Q_n}$.
\end{lem}

\begin{proof}
Let $\lambda\in\mathcal{P}(n)$ be such that $\chi=\chi^\lambda$. 
If $\lambda\neq \lambda'$ then $\phi=\chi_{\fA_n}$ and the statement clearly holds. 
Suppose that $\lambda=\lambda'$ and assume (without loss of generality) that $\phi=\phi^\lambda_+$. Let $\sigma\in P_n\smallsetminus Q_n$. Clearly $\sigma$ acts by conjugation on $\mathrm{Irr}(Q_n)$ since $Q_n\trianglelefteq P_n$. Moreover we have that $\phi^\sigma=\phi^\lambda_-$, since $\sigma\in\fS_n\smallsetminus\fA_n$. 
Therefore we have that $(\phi_{Q_n})^\sigma=(\phi^\sigma)_{Q_n}=(\phi^\lambda_-)_{Q_n}$. 
Since $(\theta_{Q_n})^\sigma=(\theta^\sigma)_{Q_n}=\theta_{Q_n}$, we deduce that $\theta_{Q_n}$ is a constituent of both $\phi_{Q_n}$ and $(\phi^\sigma)_{Q_n}$.
\end{proof}

We are now in the position to prove Theorem A of the introduction for the prime $2$ and for all natural numbers having at least $2$ even binary digits. 

\begin{prop}\label{prop:n not power of 2}
Let $n\in\mathbb{N}$ be such that $n\notin\{2^k, 2^k+1\ |\ k\in\mathbb{N}_0\}$ and let $\phi\in\mathrm{Irr}(\fA_n)$ be such that $\phi(1)$ is even. Then the restriction of $\phi$ to $P_n$ has at least $2$ 
distinct linear constituents. 
\end{prop}
\begin{proof}
Let $\lambda$ be a partition of $n$ such that $\phi$ is an irreducible constituent of $(\chi^\lambda)_{\fA_n}$. Observe that $\chi^\lambda(1)$ is even and let $\theta_1,\ldots \theta_\ell$ be the distinct linear constituents of $(\chi^\lambda)_{P_n}$. By \cite[Theorem A]{GN} we know that $\ell\geq 2$. Suppose for a contradiction that $(\theta_j)_{Q_n}=(\theta_1)_{Q_n}$ for all $j\in\{2,\ldots, \ell\}$. Then using Corollary \ref{cor:g lambda} we could pick $g\in Q_n$ such that $$0=\chi^\lambda(g)=m\theta_1(g)=\pm m,$$ where $m$ is the positive number equal to the sum of the multiplicities of the linear characters in the decomposition of $(\chi^\lambda)_{P_n}$ into irreducible constituents. This is clearly a contradiction. 
Hence we always have two linear constituents that do not coincide on $Q_n$.
%
The statement now follows from Lemma \ref{lem:lin1}.
\end{proof}

We are left to deal with the case where $n\in\{2^k, 2^k+1\}$ for some $k\in\mathbb{N}_0$. 
In order to prove Theorem A in this final specific case we must use a more combinatorial approach. 
In particular we will show that the statement holds for $n=2^k$ and derive the case $n=2^k+1$ as a corollary.

\begin{prop}\label{prop:4divides}
Let $k\in\mathbb{N}_0$, let $n=2^k+\varepsilon$ for some $\varepsilon\in\{0,1\}$, let $\lambda\in\mathcal{P}(n)$ be such that $4$ divides $\chi^\lambda(1)$ and let $\phi$ be an irreducible constituent of $(\chi^\lambda)_{\fA_n}$. Then $\phi_{Q_n}$ has at least two distinct linear constituents.
\end{prop}
\begin{proof}
Let $\gamma_k\in Q_n$ be the element of cycle type $(2^{k-1},2^{k-1}, \varepsilon)$ defined before Lemma \ref{lem:2-root}.
Lemma \ref{lem: towers8} shows that $\mathcal{H}_{2^k}(\lambda)=\emptyset$ and that $|\mathcal{H}_{2^{k-1}}(\lambda)|\leq 1$. This implies that $\chi^\lambda(\gamma_k)=0$. Suppose now for a contradiction that all linear constituents of $(\chi^\lambda)_{P_n}$ coincide on $Q_n$. Evaluation of $(\chi^\lambda)_{P_n}$ at $\gamma_k$ leads to a contradiction. The statement now follows from Lemma \ref{lem:lin1}.
\end{proof}

Proposition \ref{prop:4divides} shows that we just need to study the case where $\phi\in\mathrm{Irr}(\fA_n)$ is an even-degree constituent of $(\chi^\lambda)_{\fA_n}$ for some partition $\lambda$ such that $\nu_2(\chi^\lambda(1))=1$.
In particular, since $\phi(1)$ is even we deduce that $\lambda\neq\lambda'$.

\begin{lem}\label{lem: 2linears}
Let $\lambda\in\mathcal{P}(2^k)$ be such that $\nu_2(\chi^\lambda(1))=1$ and suppose that $(\chi^\lambda)_{Q_{2^k}}$ has a unique linear constituent $\theta$. 
Then $(\chi^\lambda)_{Q_{2^k}}=2\theta+\Delta$, where $\Delta$ is a sum (possibly empty) of even degree irreducible characters of $Q_{2^k}$. In particular $(\chi^\lambda)_{P_{2^k}}$ has exactly two linear constituents both appearing with multiplicity $1$. 
\end{lem}
\begin{proof}
Let $(\chi^\lambda)_{P_{2^k}}=c_1\theta_1+\cdots+c_\ell\theta_\ell+\Omega$, where $\theta_j\in\mathrm{Lin}(P_{2^k})$, $c_j\in\mathbb{N}$ for all $j\in\{1,\ldots, \ell\}$ and $\Omega$ is a sum (possibly empty) of even degree irreducible characters of $P_{2^k}$. From \cite[Theorem A]{GN} we know that $\ell\geq 2$. From the hypothesis we get that $(\theta_j)_{Q_{2^k}}=\theta$ for all  $j\in\{1,\ldots, \ell\}$. Since $\theta^{P_{2^k}}(1)=2$, it follows that $\ell=2$ and that  $(\chi^\lambda)_{P_{2^k}}=c_1\theta_1+c_2\theta_2+\Omega$. 
Since $\nu_2(\chi^\lambda(1))=1$ we deduce from Lemma \ref{lem: towers8} that $\mathcal{H}_{2^k}(\lambda)=\emptyset$ and that $|\mathcal{H}_{2^{k-1}}(\lambda)|=2$. 
In particular we have that $C_{2^{k-1}}(\lambda)=\emptyset$, $w_{2^{k-1}}(\lambda)=2$ and $Q_{2^{k-1}}(\lambda)$ is a sequence of $2^{k-1}$ partitions such that all except for two are empty. The two non empty partitions are both equal to $(1)$, the unique partition of $1$. 
Using Lemma \ref{lem:farahat} we observe that $\chi^\lambda(\gamma_{k})=\pm 2$. 
On the other hand we have that $\chi^\lambda(\gamma_{k})=\pm (c_1+c_2)$, by Lemma \ref{lem:2-root}. We obtain that $c_1=c_2=1$. 
This concludes the proof. 
\end{proof}

\begin{prop}\label{cor: 2k}
Let $\lambda\in\mathcal{P}(2^k)$ be such that $\nu_2(\chi^\lambda(1))=1$ and that $\lambda\neq\lambda'$. Then $(\chi^\lambda)_{Q_{2^k}}$ has at least two distinct linear constituents. 
\end{prop}

The proof of Proposition \ref{cor: 2k} is quite technical. For this reason we decided to postpone it. Section \ref{sec:4} below is entirely devoted to prove this statement.

\begin{cor}\label{cor: 2k+1}
Let $n\in\{2^k, 2^{k}+1\}$, for some $k\in\mathbb{N}$ and let $\phi\in\mathrm{Irr}(\fA_n)$ be such that $\phi(1)$ is even. Then $\phi_{Q_{n}}$ has two distinct linear constituents. 
\end{cor}
\begin{proof}
Let $\lambda\in\mathcal{P}(n)$ be such that $\phi$ is an irreducible constituent of $(\chi^\lambda)_{\fA_n}$.
Proposition \ref{prop:4divides} and Proposition \ref{cor: 2k} show that it is enough to consider the case where $n=2^k+1$ and where $\nu_2(\chi^\lambda(1))=1$.
Since $\phi(1)$ is even we immediately get that $\lambda\neq\lambda'$ and hence that $\phi=(\chi^\lambda)_{\fA_n}$.
Moreover, in this situation we have that $\mathcal{H}_{2^k}(\lambda)=\emptyset$ and that $|\mathcal{H}_{2^{k-1}}(\lambda)|=2$. It follows that $\lambda$ is not a hook partition of $2^{k}+1$. Therefore there exists a partition $\mu\in\mathcal{P}(2^k)\smallsetminus\mathcal{L}(2^k)$ such that $\chi^\mu$ is an irreducible constituent of $(\chi^\lambda)_{\fS_{n-1}}$. Since $\mu$ is not a hook partition we have that $\chi^\mu(1)$ is even. Hence $(\chi^\mu)_{Q_{n-1}}$ has two distinct linear constituents by Propositions \ref{prop:4divides} and \ref{cor: 2k}. 
Since $Q_n=Q_{n-1}$ we deduce that $(\chi^\mu)_{Q_{n-1}}$ is a constituent of $\phi_{Q_n}$ and the statement follows. 
\end{proof}

We conclude by observing that Propositions \ref{prop:n not power of 2} and Corollary \ref{cor: 2k+1} show that Theorem A of the introduction holds.

\section{The proof of Proposition \ref{cor: 2k}}\label{sec:4}

In this section we focus on the proof of Proposition \ref{cor: 2k}. This is certainly the most technical part of the article. 
We start by recalling a couple of preliminary results. 
The first one was proved in \cite[Theorem 1.1]{G}. 

\begin{thm}\label{thm: JLMS}
Let $h\in\mathcal{L}(2^k)$ and let $\chi^h\in\mathrm{Irr}_{2'}(\fS_{2^k})$. Then $(\chi^h)_{P_{2^k}}$ has a unique linear constituent $f(h)$. Moreover, $\chi^h\mapsto f(h)$ is a bijection between $\mathrm{Irr}_{2'}(\fS_{2^k})$ and $\mathrm{Lin}(P_{2^k})$.
\end{thm}

In \cite{G3} it was introduced the concept of $2$-sections of a partition. For the proof of Proposition \ref{cor: 2k} we will only need the following specific instance of that more general definition. The extended concept of $q$-sections of partitions was the key idea to prove Theorem A of \cite{GN} for symmetric groups. 

\begin{notation}\label{not: Delta}
Let $n$ be a natural number. Given a composition $\mu$ of $n$ we denote by $\mu^\star$ the partition obtained from $\mu$ by reordering its parts. Let $\lambda$ be a partition of $2n$. 
We can uniquely write $\lambda$ as $\lambda=(\omega\circ \varepsilon)^\star$ where $\omega$ is the partition consisting of all the odd parts of $\lambda$ and $\varepsilon$ is the partition consisting of all the even parts of $\lambda$. 
(Here the symbol $\circ$ denotes the concatenation of partitions.)
In particular we have that $$\lambda=\big((2k_1+1,\ldots, 2k_t+1)\circ (2r_1,\ldots ,2r_s)\big)^\star,\ \ \ \ \ $$
where $k_1\geq \cdots\geq k_t\geq 0$ and $r_1\geq \cdots\geq r_s> 0$.
Since $\lambda\vdash 2n$ we clearly have that $t=2\zeta$ for some $\zeta\in\mathbb{N}_{0}$. 
We let $\Delta^2(\lambda)$ be the partition of $n$ defined by 
$$\Delta^2(\lambda)=\big((k_1+1, \ldots, k_\zeta+1,  k_{\zeta+1}, \ldots, k_t)\circ (r_1,\ldots ,r_s)\big)^\star.$$
\end{notation}

In  \cite[Theorem 3.9]{GN} the following was proved. 

\begin{thm} \label{thm: Delta}
Let $n$ be a natural number and let $\lambda$ be a partition of $2n$. 
The irreducible character $\chi^{\Delta^2(\lambda)}\times\chi^{\Delta^2(\lambda)}$ is an irreducible constituent of $(\chi^\lambda)_{\fS_{n}\times\fS_n}$.
\end{thm}

We are ready to prove Proposition \ref{cor: 2k}. For convenience, we recall its statement here.

\begin{prop}
Let $\lambda\in\mathcal{P}(2^k)$ be such that $\nu_2(\chi^\lambda(1))=1$ and that $\lambda\neq\lambda'$. Then $(\chi^\lambda)_{Q_{2^k}}$ has at least two distinct linear constituents. 
\end{prop}
 \begin{proof}
To ease the notation we let $q:=2^{k-1}$.
From Lemma \ref{lem: towers8} we know that $\nu_2(\chi^\lambda(1))=1$ if and only if $\mathcal{H}_{2^k}(\lambda)=\emptyset$ and $|\mathcal{H}_{q}(\lambda)|=2$. Hence $\lambda$ is a partition of $2^k$ having exactly two hooks of length $q$. We immediately see that the only possibility for $\lambda$ is to have $h_{1,2}(\lambda)=h_{2,1}(\lambda)=q$. It follows that there exist $x\in\{1,2,\ldots, q-1\}$ and $y\in\{x-1,\ldots, q-2\}$ such that $$\lambda=(q-x+1, q-y, 2^{x-1}, 1^{y-(x-1)}).$$
The best way to proceed is to distinguish two cases, depending on the parity of the number $s:=y-(x-1)$. 

\textbf{(i)} Suppose that $s$ is even and let 
$$\mu=(q-y+\frac{s}{2}, 1^{y-\frac{s}{2}})\ \ \text{and}\ \ \  \nu=(q-y+(\frac{s}{2}-1), 1^{y-\frac{s}{2}+1}).$$
An easy calculation shows that $\mu\notin\{\nu, \nu'\}$ since $\lambda\neq \lambda'$. 
We will now show that $\chi^\rho\times\chi^\rho$ is an irreducible constituent of $(\chi^\lambda)_{\fS_q\times\fS_q}$, for all $\rho\in\{\mu, \nu\}$. 

Notice that $y-s/2=x+s/2-1$.
It is easy to see that the skew Young diagram $[\lambda\smallsetminus\mu]$ is the disjoint union of the three disconnected diagrams $Z_1:=H_{1, q-y+\frac{s}{2}+1}(\lambda)$, 
$Z_2:=H_{2,2}(\lambda)$ and $Z_3:=H_{x+1+\frac{s}{2},1}(\lambda)$. 
Observe that $Z_1$ is a row of shape $(s/2)$, $Z_2$ is a hook of shape $(q-y-1, 1^{x-1})$ and $Z_3$ is a column of shape $(1^{s/2+1})$. 
With this in mind, we proceed as follows. 

\medskip

\textbf{-} Replace each of the $s/2$ nodes of $Z_1$ by $1$. 

\textbf{-} Replace each of the $q-y-1$ nodes of the first row of the hook $Z_2$ by $1$ and 

replace the nodes in the leg of $Z_2$ increasingly from top to bottom with the numbers 

$2,3,\ldots, x$. 

\textbf{-} Replace the top node of the column $Z_3$ by $1$ and replace the remaining $s/2$ 

nodes 
increasingly from top to bottom 
with the 
numbers 
$x+1,x+2,\ldots, x+s/2$. 

\medskip 

\noindent What we obtain is a Littlewood-Richardson configuration of type $\mu$ for $[\lambda\smallsetminus\mu]$ because the numbers are weakly increasing along rows, stricly increasing along columns and the sequence obtained by reading the numbers right to left and top to bottom is a reverse lattice sequence of type $\mu$. This follows by observing that the number of $1$s in the sequence is exactly $q-y+s/2=\mu_1$.  Moreover, for all $j\in\{2,3,\ldots, x+s/2\}$ we have exactly one $j$. The claim now follows since $x+s/2-1=y-s/2$. 
We remark that the process described above works in the case where $s=0$ as well ($Z_1$ is empty in this case). 

We conclude that $\chi^\mu\times\chi^\mu$ is an irreducible constituent of $(\chi^\lambda)_{\fS_q\times\fS_q}$.

\smallskip

We proceed in a very similar way to show that $\chi^\nu\times\chi^\nu$ is an irreducible constituent of $(\chi^\lambda)_{\fS_q\times\fS_q}$.
We observe that the skew Young diagram $[\lambda\smallsetminus\nu]$ is the disjoint union of the three disconnected diagrams $W_1:=H_{1, q-y+\frac{s}{2}}(\lambda)$, 
$W_2:=H_{2,2}(\lambda)$ and $W_3:=H_{x+2+\frac{s}{2},1}(\lambda)$. 
Observe that $W_1$ is a row of shape $(s/2+1)$, $W_2$ is a hook of shape $(q-y-1, 1^{x-1})$ and $W_3$ is a column of shape $(1^{s/2})$. 
With this in mind, we proceed as follows. 

\medskip

\textbf{-} Replace each of the $s/2+1$ nodes of $W_1$ by $1$. 

\textbf{-} Replace node $(2,q-y)$ (this is the most right node of the first row of the hook 

$W_2$) 
with a $2$ and replace each of the remaining $q-y-2$ nodes in the first row of 

$W_2$ by $1$. Moreover, 
replace the nodes in the leg of $W_2$ increasingly from top to 

bottom with the numbers 
$3,4, \ldots, x+1$. 

\textbf{-} Replace the $s/2$ 
nodes of the column $W_3$
increasingly from top to bottom 
with the 

numbers 
$x+2,\ldots, x+s/2+1$. 

\medskip 

Arguing exactly as above we verify that this process leaves us with a Littlewood Richardson configuration of type $\nu$ for $[\lambda\smallsetminus\nu]$. Again we remark that the steps described above work in the case where $s=0$ (in this case we have that $Z_3$ is empty). 
We conclude that $\chi^\nu\times\chi^\nu$ is an irreducible constituent of $(\chi^\lambda)_{\fS_q\times\fS_q}$.

\medskip

Using Theorem \ref{thm: JLMS} we observe that both $f(\mu)\times f(\mu)$ and $f(\nu)\times f(\nu)$ are irreducible constituents of $(\chi^\lambda)_{P_q\times P_q}$.
Hence for all $\rho\in\{\mu, \nu\}$ there exists an extension $\theta_\rho\in\mathrm{Lin}(P_{2^k})$ of $f(\rho)\times f(\rho)$ such that $\theta_\rho$ is a linear constituent of $(\chi^\lambda)_{P_{2^k}}$. Since $\mu\notin\{\nu, \nu'\}$ we immediately deduce that $(\theta_\mu)_{Q_{2^k}}\neq (\theta_\nu)_{Q_{2^k}}$. Hence the proof is concluded in this case. 

\bigskip

\textbf{(ii)} Suppose that $s$ is odd. In this case it is not possible to find two distinct hook partitions $\mu, \nu$ of $q$ such that $\chi^\rho\times \chi^\rho$  is an irreducible constituent of $(\chi^\lambda)_{\fS_q\times\fS_q}$, for all $\rho\in\{\mu, \nu\}$ (take for instance $\lambda=(4,3,1)$ to see this in a small example). 
Hence we need to proceed with a sligthly different argument. Let
$$\mu=(q-y+\frac{s-1}{2}, 1^{y-\frac{s-1}{2}})\ \ \text{and}\ \ \  \nu=\Delta^2(\lambda),$$
where $\Delta^2$ is the operator defined in Notation \ref{not: Delta}.
We will now show that $\chi^\rho\times\chi^\rho$ is an irreducible constituent of $(\chi^\lambda)_{\fS_q\times\fS_q}$, for all $\rho\in\{\mu, \nu\}$. 
We know from Theorem \ref{thm: Delta} that $\chi^\nu\times\chi^\nu$ is an irreducible constituent of $(\chi^\lambda)_{\fS_q\times\fS_q}$. 

In order to show that the same holds for $\mu$ we proceed with a very similar strategy to the one used in the previous case. 
Notice that $y-(s-1)/2=x+(s-1)/2$.
It is easy to see that the skew Young diagram $[\lambda\smallsetminus\mu]$ is the disjoint union of the three disconnected diagrams $V_1:=H_{1, q-y+\frac{s+1}{2}}(\lambda)$, 
$V_2:=H_{2,2}(\lambda)$ and $V_3:=H_{x+1+\frac{s+1}{2},1}(\lambda)$. 
Observe that $V_1$ is a row of shape $((s+1)/2)$, $V_2$ is a hook of shape $(q-y-1, 1^{x-1})$ and $V_3$ is a column of shape $(1^{(s+1)/2})$. 
With this in mind, we proceed as follows. 

\medskip

\textbf{-} Replace each of the $(s+1)/2$ nodes of $V_1$ by $1$. 

\textbf{-} Replace each of the $q-y-1$ nodes of the first row of the hook $V_2$ by $1$ and 

replace the nodes in the leg of $V_2$ increasingly from top to bottom with the numbers 

$2,3,\ldots, x$. 

\textbf{-} Replace the $(s+1)/2$ 
nodes of the column $W_3$
increasingly from top to bottom 

with the 
numbers 
$x+1,\ldots, x+(s+1)/2$. 

\smallskip 

Arguing exactly as before we deduce that the process described above gives a Littlewood -Richardson configuration of type $\mu$ for $[\lambda\smallsetminus \mu]$. 
Hence we obtain that $\chi^\mu\times \chi^\mu$  is an irreducible constituent of $(\chi^\lambda)_{\fS_q\times\fS_q}$.

\medskip

Whenever $\nu=\Delta^2(\lambda)\notin\mathcal{L}(q)$ then we have that $\chi^\nu(1)$ is even (see the discussion after Theorem \ref{thm:oddMac}) and therefore we know from \cite[Theorem A]{GN} that $(\chi^\nu)_{P_q}$ admits two distinct linear constituents $\eta_1, \eta_2$. Moreover, using again \cite[Theorem 1.1]{G} we know that $f(\mu)\in\mathrm{Lin}(P_q)$ is the unique linear constituent of $(\chi^\mu)_{P_q}$. This shows that $(\chi^\lambda)_{P_q\times P_q}$ has at least $3$ (not necessarily distinct) linear constituents of the form $\phi\times\phi$, for some $\phi\in\mathrm{Lin}(P_q)$. It follows that $(\chi^\lambda)_{P_{2^k}}$ has at least $3$ (not necessarily distinct) linear constituents. Since $\nu_2(\chi^\lambda(1))=1$ we can now use Lemma \ref{lem: 2linears} to conclude that $(\chi^\lambda)_{Q_{2^k}}$ can not have a unique linear constituent. 

\medskip

If $\Delta^2(\lambda)\in\mathcal{L}(q)$ we have that $\lambda_2\in\{2,3\}$ by the construction described in Notation \ref{not: Delta}. Suppose for a contradiction that $\lambda_2=3$. Then we have that $y$ and $x$ are odd. Hence $\lambda_1$ is even and therefore $(\Delta^2(\lambda))_2=2$. Hence $\Delta^2(\lambda)\notin\mathcal{L}(q)$. It follows that $\lambda_2=2$ and hence that $y=q-2$. In particular $s=q-x-1$ and we can write $\lambda$ as
$\lambda=(s+2, 2^x, 1^s).$
Consider the conjugate partition $\lambda'$. We certainly have that $\nu_2(\chi^{\lambda'}(1))=1$. Moreover it is easy to see that $\lambda'=(q, x+1, 1^s)$. Since $(\lambda')_2=x+1\geq 3$ and $s>0$ we have that $\Delta^2(\lambda')\notin\mathcal{H}(q)$. Therefore we deduce that $\lambda'$ is one of the partitions we already considered above in case (ii). It follows that $(\chi^{\lambda'})_{P_{2^k}}$ has at least $3$ (not necessarily distinct) linear constituents. Since $(\chi^{\lambda})_{Q_{2^k}}=(\chi^{\lambda'})_{Q_{2^k}}$ the proof is now concluded, by Lemma \ref{lem: 2linears}. 
\end{proof}

\medskip


\begin{thebibliography}{9999}

\bibitem{G}
{\sc E. Giannelli,}
\newblock Characters of odd degree of symmetric groups, 	
\newblock \emph{J. London Math. Soc.} \textbf{(1)}, 96 (2017), 1--14.

\bibitem{G3}
{\sc E. Giannelli,}
\newblock On the restriction of irreducible characters of symmetric groups to Sylow $p$-subgroups, 	
\newblock \emph{J. Algebra} \textbf{(483)}, (2017), 37--57.

\bibitem{GN}
{\sc  E.~Giannelli and G.~Navarro, }
\newblock Restricting irreducible characters to Sylow $p$-subgroups, 
\newblock to appear in \emph{Proc. Amer. Math. Soc.} DOI: 10.1090/proc/13970.

\bibitem{I}
{\sc I. M. Isaacs,} `{\it Character Theory of Finite Groups}',
AMS-Chelsea Publishing, Providence, 2006.

\bibitem{James}
{\sc  G.~D. James, }
\newblock \emph{The representation theory of the symmetric groups}, 
\newblock Lecture Notes in Mathematics, vol. 682, Springer, Berlin, 1978. 

\bibitem{JK}
{\sc  G.~James and A.~Kerber, }
\newblock \emph{The representation theory of the symmetric
  group}, 
  \newblock Encyclopedia of Mathematics and its Applications, vol.~16,
  Addison-Wesley Publishing Co., Reading, Mass., 1981. 

\bibitem{Mac}
{\sc  I. G. Macdonald,}
\newblock On the degrees of the irreducible representations of symmetric groups,
\newblock {\em Bull. London Math. Soc.} 3 1971 189–192. 

\bibitem{Malle}
{\sc  G. Malle,}
\newblock Local-global conjectures in the representation theory of finite groups. 
\newblock {\it Representation theory - current trends and perspectives.} EMS Ser. Congr. Rep , Eur. Math. Soc., Zurich (2017), 519-539

\bibitem{McKay}
{\sc  J. McKay,}
\newblock Irreducible representations of odd degree.
\newblock {\em J. Algebra}
{\bf 20} (1972), 416--418  

\bibitem{N}
{ \sc G. Navarro,}
\newblock Linear characters of Sylow subgroups, 
\newblock {\em J. Algebra} {\bf 269} (2003), no. 2, 589--598. 

\bibitem{OlssonBook}
{\sc  J. Olsson, }
\newblock \emph{Combinatorics and Representations of Finite Groups}, 
  \newblock Vorlesungen aus dem Facherbeich Mathematik der Universitat Essen, Heft 20, 1994.   


 



 
\end{thebibliography}
\end{document}